\documentclass[12pt,a4paper]{article}
\usepackage{amsfonts}
\usepackage{amsmath}
\usepackage{graphicx}
\usepackage{boxedminipage}

\newtheorem{theorem}{Theorem}

\newtheorem{corollary}[theorem]{Corollary}

\newtheorem{lemma}[theorem]{Lemma}

\newenvironment{proof}[1][Proof]{\noindent\textbf{#1.} }{\ \rule{0.5em}{0.5em}}

\begin{document}

\title{On Koksma-Hlawka inequality}
\author{L. Brandolini \and L. Colzani \and G. Gigante \and G. Travaglini}
\date{}
\maketitle

\begin{abstract}
The classical Koksma Hlawka inequality does not apply to functions with
simple discontinuities. Here we state a Koksma Hlawka type inequality which
applies to piecewise smooth functions $f\chi _{\Omega }$, with $f$ smooth
and $\Omega $ a Borel subset of $\left[ 0,1\right] ^{d}$:

\begin{equation*}
\left\vert N^{-1}\sum_{j=1}^{N}\left( f\chi _{\Omega }\right) \left(
x_{j}\right) -\int_{\Omega }f(x)dx\right\vert \leq \mathcal{D}\left( \Omega
,x_{j}\right) \mathcal{V}(f),
\end{equation*}%
where $\mathcal{D}\left( \Omega ,x_{j}\right) $ is the discrepancy
\begin{equation*}
\mathcal{D}\left( \Omega ,x_{j}\right) =2^{d}\sup\limits_{I\subseteq \left[
0,1\right] ^{d}}\left\{ \left\vert N^{-1}\sum_{j=1}^{N}\chi _{\Omega \cap
I}\left( x_{j}\right) -\left\vert \Omega \cap I\right\vert \right\vert
\right\} ,
\end{equation*}%
the supremum is over all $d$-dimensional intervals, and $\mathcal{V}(f)$ is
the total variation
\begin{equation*}
\mathcal{V}(f)=\sum_{\alpha \in \left\{ 0,1\right\} ^{d}}2^{d-\left\vert
\alpha \right\vert }\int_{\left[ 0,1\right] ^{d}}\left\vert \left( \frac{%
\partial }{\partial x}\right) ^{\alpha }f\left( x\right) \right\vert dx.
\end{equation*}%
We state similar results with variation and discrepancy measured by $L^{p}$
and $L^{q}$ norms, $1/p+1/q=1$, and we also give extensions to compact
manifolds.
\end{abstract}

\section{Introduction}

Koksma's inequality is a neat bound for the error in a numerical
integration:
\begin{equation*}
\left| N^{-1}\sum_{j=1}^{N}f\left( x_{j}\right) -\int_{0}^{1}f(x)dx\right|
\leq \mathcal{D}\left( x_{j}\right) \mathcal{V}\left( f\right) .
\end{equation*}

In this inequality $\mathcal{D}\left( x_{j}\right) $ is the discrepancy of
the points $0\leq x_{j}\leq 1$ and $\mathcal{V}\left( f\right) $ is the
total variation of the function $f$,
\begin{gather*}
\mathcal{D}\left( x_{j}\right) =\sup_{0\leq t\leq 1}\left\{ \left|
N^{-1}\sum_{j=1}^{N}\chi _{\left[ 0,t\right] }\left( x_{j}\right) -t\right|
\right\} , \\
\mathcal{V}\left( f\right) =\sup_{0=y_{0}<y_{1}<y_{2}<\ldots
<y_{n}=1}\left\{ \sum_{k=1}^{n}\left| f\left( y_{k}\right) -f\left(
y_{k-1}\right) \right| \right\} .
\end{gather*}

We may say that Koksma's inequality is a simple machine\ which turns the
discrepancy for a small family of functions, characteristic functions of
intervals, into the discrepancy for a larger family, functions of bounded
variation. The extension to several variables is a more delicate problem,
yet it is of some relevance in numerical analysis. See e.g. \cite{Hick},
\cite{Hla}, \cite{KN}, \cite{Mat}, \cite{Nie}, \cite{Z}. A classical
approach starts with the definitions of Vitali and Hardy Krause variations.
For a function $f$ on $\left[ 0,1\right] ^{d}$ and for every $d$ dimensional
interval $I$ in $\left[ 0,1\right] ^{d}$ with edges parallel to the axes,
let $\Delta \left( f,I\right) $ be an alternating sum of the values of $f$
at the vertices of $I$. The Vitali variation is
\begin{equation*}
V\left( f\right) =\sup_{R}\left\{ \sum_{I\in R}\left| \Delta \left(
f,I\right) \right| \right\} ,
\end{equation*}
where the supremum is over all finite partitions $R$ of $\left[ 0,1\right]
^{d}$ in $d$ dimensional intervals $I$. The Hardy Krause variation is
\begin{equation*}
\mathcal{V}\left( f\right) =\sum_{k}V_{k}\left( f\right) ,
\end{equation*}
where the sum is over the Vitali variations $V_{k}\left( f\right) $ of the
restrictions of $f$ to all faces of all dimensions of $\left[ 0,1\right]
^{d} $. The discrepancy of a finite point set $\left\{ x_{j}\right\}
_{j=1}^{N}$ in $\left[ 0,1\right] ^{d}$ is defined by
\begin{equation*}
\mathcal{D}\left( x_{j}\right) =\sup_{I}\left\{ \left|
N^{-1}\sum_{j=1}^{N}\chi _{I}\left( x_{j}\right) -\left| I\right| \right|
\right\} ,
\end{equation*}
where $I$ is an interval of the form $\left[ 0,t_{1}\right] \times \left[
0,t_{2}\right] \times \ldots \times \left[ 0,t_{d}\right] $ with $0\leq
t_{k}\leq 1$, and $\left| I\right| =t_{1}t_{2}\cdots t_{d}$ is its measure.
The classical Koksma Hlawka inequality states that if $f$ has bounded Hardy
Krause variation, then
\begin{equation*}
\left| N^{-1}\sum_{j=1}^{N}f\left( x_{j}\right) -\int_{\left[ 0,1\right]
^{d}}f(x)dx\right| \leq \mathcal{D}\left( x_{j}\right) \mathcal{V}\left(
f\right) .
\end{equation*}

The assumptions required in the one dimensional Koksma inequality are
satisfied by many familiar functions and are usually easy to verify. On the
contrary, the Hardy Krause condition in the Koksma Hlawka inequality seems
to be rather strict. It works well for smooth functions, but it cannot be
applied to most functions with simple discontinuities. For example, the
characteristic function of a convex polyhedron has bounded Hardy Krause
variation only if the polyhedron is a $d$ dimensional interval. For this and
other reasons, several variants of the Koksma Hlawka inequality have been
proposed. In particular, in \cite{Har10} the small family consists of
characteristic functions of convex sets and the large family is given by
functions with super level sets which are differences of finite unions of
convex sets. See also \cite[p.162]{Har07}, \cite{Mat}, \cite{Tra}. Finally,
a general and systematic approach to Koksma Hlawka inequalities is via
reproducing kernel Hilbert spaces. See e.g. \cite{AZ} and \cite{Hick}.
However, in some of these approaches the geometric meaning of the
discrepancy is somehow hidden. The aim of this paper is to state some Koksma
Hlawka inequalities with explicit geometric discrepancies, and which apply
to piecewise smooth functions, that is smooth functions $f$ restricted to
arbitrary Borel sets $\Omega $. In one version of this inequality the error
in the numerical integration of $f\chi _{\Omega }$ is controlled by a
variation of $f$ defined in terms of derivatives, times the discrepancy of
the intersection of $\Omega $ with translates of intervals $I$ with edges
parallel to the axes. In another version the discrepancy is with respect to
the intersection of $\Omega $ with cubes, and in a further version the
discrepancy is with respect to the intersection of $\Omega $ with balls.
These results are first stated and proved when the underlying space is a
torus, then they are extended to compact manifolds, in particular spheres.

\section{Koksma Hlawka inequalities on a torus}

In what follows we are going to consider functions, measures, distributions,
on the torus $\mathbb{T}^{d}=\mathbb{R}^{d}/\mathbb{Z}^{d}=\left[ 0,1\right)
^{d}$, that is functions, measures, distributions on $\mathbb{R}^{d}$ which
are $\mathbb{Z}^{d}$ periodic.

\begin{theorem}
\label{1}If $f$ and $\mu $ are an integrable function and a finite measure
on $\mathbb{T}^{d}$ respectively, if $\Omega $ is a bounded Borel subset of $%
\mathbb{R}^{d}$, and if $1\leq p,q\leq +\infty $ with $1/p+1/q=1$, then
\begin{equation*}
\left\vert \int_{\Omega }f(x)\overline{d\mu (x)}\right\vert \leq \mathcal{D}%
_{q}\left( \Omega ,\mu \right) \mathcal{V}_{p}(f),
\end{equation*}%
where $\mathcal{D}_{q}\left( \Omega ,\mu \right) $ is the $L^{q}$
discrepancy
\begin{equation*}
\mathcal{D}_{q}\left( \Omega ,\mu \right) =\int_{\left[ 0,1\right]
^{d}}\left\{ \int_{\mathbb{T}^{d}}\left\vert \sum_{n\in \mathbb{Z}^{d}}\mu
\left( \left( x+n-I\left( t\right) \right) \cap \Omega \right) \right\vert
^{q}dx\right\} ^{1/q}dt,
\end{equation*}%
with $I(t)=\left[ 0,t_{1}\right] \times \left[ 0,t_{2}\right] \times \ldots
\times \left[ 0,t_{d}\right] $, $0\leq t_{k}\leq 1$, and $\mathcal{V}_{p}(f)$
is the $L^{p}$ total variation
\begin{equation*}
\mathcal{V}_{p}(f)=\sum_{\alpha \in \left\{ 0,1\right\} ^{d}}2^{d-\left\vert
\alpha \right\vert }\left\{ \int_{\mathbb{T}^{d}}\left\vert \left( \frac{%
\partial }{\partial x}\right) ^{\alpha }f\left( x\right) \right\vert
^{p}dx\right\} ^{1/p}.
\end{equation*}
\end{theorem}

The variation of the function and the discrepancy of the measure in the
statement of the theorem increase with $p$ and $q$, hence a gain in $q$
corresponds to a loss in $p$. When $\Omega $ is contained in $\left[ 0,1%
\right] ^{d}$ the $L^{\infty }$ discrepancy is dominated as follows:
\begin{equation*}
\mathcal{D}_{\infty }\left( \Omega ,\mu \right) \leq
2^{d}\sup\limits_{I\subseteq \left[ 0,1\right] ^{d}}\left\{ \left\vert \mu
\left( I\cap \Omega \right) \right\vert \right\} .
\end{equation*}

This reflects the difference between the discrepancy in a torus and the one
in a cube, and it is due to the fact that an interval in $\mathbb{T}^{d}$
can be split into at most $2^{d}$ intervals in $\left[ 0,1\right] ^{d}$. In
the above formulas with $p=1$ the derivatives can be measures, and in this
case the norms $\int_{\mathbb{T}^{d}}\left| \left( \partial /\partial
x\right) ^{\alpha }f\left( x\right) \right| dx$ denote the total variations.
Observe that if $p=1$ then less than $d$ integrable derivatives are not
enough to guarantee the boundedness of the functions. Hence, if the measure $%
\mu $ is concentrated on the singularities of the function $f$, the integral
$\int_{\Omega }f(x)d\mu (x)$ may be not defined. In the classical Koksma
Hlawka inequality, and in most of discrepancy theory, the measure $\mu $ is
the difference between masses $\delta _{x_{j}}$ concentrated at the points $%
x_{j}$ and the uniformly distributed measure $dx$:
\begin{equation*}
d\mu =N^{-1}\sum_{j=1}^{N}\delta _{x_{j}}-dx.
\end{equation*}

For this measure and when $p=1$ and $\Omega =\left[ 0,1\right] ^{d}$ the
above theorem is essentially equivalent to the classical Koksma Hlawka
inequality with respect to the Hardy Krause variation. The proof of the
theorem can be split into a sequence of easy lemmas. The first one can be
seen as a Fourier analog of a multidimensional integration by parts in \cite%
{Z}. See also the examples in \cite{AZ}.

\begin{lemma}
\label{2}Let $\varphi $ be a non vanishing complex sequence on $\mathbb{Z}%
^{d}$, and assume that both $\varphi $ and $1/\varphi $ have tempered growth
in $\mathbb{Z}^{d}$. Also let $f$ be an integrable function on $\mathbb{T}%
^{d}$. Define
\begin{gather*}
g\left( x\right) =\sum_{n\in \mathbb{Z}^{d}}\overline{\varphi \left(
n\right) ^{-1}}e^{2\pi in\cdot x}, \\
\mathfrak{D}f\left( x\right) =\sum_{n\in \mathbb{Z}^{d}}\varphi \left(
n\right) \widehat{f}\left( n\right) e^{2\pi in\cdot x}.
\end{gather*}%
Finally, let $\mu $ be a finite measure on $\mathbb{T}^{d}$. Then, the
following identity holds:
\begin{equation*}
\int_{\mathbb{T}^{d}}f(x)\overline{d\mu (x)}=\int_{\mathbb{T}^{d}}\mathfrak{D%
}f\left( x\right) \overline{g\ast \mu \left( x\right) }dx.
\end{equation*}
\end{lemma}

\begin{proof}
We are using the notation $\widehat{f}\left( n\right) =\int_{\mathbb{T}%
^{d}}f(x)e^{-2\pi in\cdot x}dx$ for the Fourier transform and $g\ast \mu
(x)=\int_{\mathbb{T}^{d}}g(x-y)d\mu (y)$ for the convolution, and we are
applying these operators also to distributions. In particular, the
assumptions on the growth of $\varphi $ and $1/\varphi $ guarantee that both
$\mathfrak{D}f$ and $g$ are well defined as tempered distributions. Moreover
\begin{gather*}
\int_{\mathbb{T}^{d}}f(x)\overline{d\mu (x)}=\sum_{n\in \mathbb{Z}^{d}}%
\widehat{f}\left( n\right) \overline{\widehat{\mu }\left( n\right) } \\
=\sum_{n\in \mathbb{Z}^{d}}\left( \varphi \left( n\right) \widehat{f}\left(
n\right) \right) \left( \varphi \left( n\right) ^{-1}\overline{\widehat{\mu }%
\left( n\right) }\right) =\int_{\mathbb{T}^{d}}\mathfrak{D}f\left( x\right)
\overline{g\ast \mu \left( x\right) }dx.
\end{gather*}
\end{proof}

Suitable choices of $\varphi $ and $\mu $ will make the above abstract lemma
more concrete and interesting. In particular, $\varphi $ will be the Fourier
transform of a differential integral operator and $1/\varphi $ the Fourier
transform of a superposition of characteristic functions.

\begin{lemma}
\label{3}Let the function $h$ on $\mathbb{R}^{d}$ be the superposition of
all intervals $I(t)=\left[ 0,t_{1}\right] \times \left[ 0,t_{2}\right]
\times \ldots \times \left[ 0,t_{d}\right] $ with $0\leq t_{k}\leq 1$, and
let $g(x)$\textit{\ be the }$\mathbb{Z}^{d}$ \textit{periodization of }$h$,
\begin{equation*}
h(x)=\int_{\left[ 0,1\right] ^{d}}\chi _{I\left( t\right) }\left( x\right)
dt,\ \ \ g\left( x\right) =\sum_{n\in \mathbb{Z}^{d}}h\left( x+n\right) .
\end{equation*}%
Then the function $g$ has Fourier expansion
\begin{equation*}
g\left( x\right) =\sum_{n\in \mathbb{Z}^{d}}\left( \prod_{k=1}^{d}\left(
2\delta \left( n_{k}\right) +2\pi in_{k}\right) ^{-1}\right) e^{2\pi in\cdot
x},
\end{equation*}%
where $\delta \left( 0\right) =1$ and $\delta \left( j\right) =0$ for $j\neq
0$.
\end{lemma}

\begin{proof}
Observe that
\begin{equation*}
h(x)=\prod_{k=1}^{d}\int_{0}^{1}\chi _{\left[ 0,t_{k}\right] }\left(
x_{k}\right) dt_{k}=\prod_{k=1}^{d}\left( 1-x_{k}\right) \chi _{\left[ 0,1%
\right] }\left( x_{k}\right) .
\end{equation*}

Then compute the Fourier coefficients,
\begin{gather*}
\widehat{g}\left( n\right) =\int_{\mathbb{T}^{d}}g(x)e^{-2\pi in\cdot x}dx \\
=\prod_{k=1}^{d}\left( \int_{0}^{1}\left( 1-x_{k}\right) e^{-2\pi
in_{k}x_{k}}dx_{k}\right) =\prod_{k=1}^{d}\left( 2\delta \left( n_{k}\right)
+2\pi in_{k}\right) ^{-1}.
\end{gather*}
\end{proof}

\begin{lemma}
\label{4}If $f$ is a smooth function on $\mathbb{T}^{d}$, then
\begin{gather*}
\mathfrak{D}f(x)=\sum_{n\in \mathbb{Z}^{d}}\left( \prod_{k=1}^{d}\left(
2\delta \left( n_{k}\right) -2\pi in_{k}\right) \right) \widehat{f}\left(
n\right) \exp \left( 2\pi inx\right) \\
=\sum_{\alpha ,\beta \in \left\{ 0,1\right\} ^{d},\;\alpha +\beta =\left(
1,\ldots ,1\right) }\left( -\right) ^{\left\vert \alpha \right\vert
}2^{\left\vert \beta \right\vert }\int_{\left[ 0,1\right] ^{\left\vert \beta
\right\vert }}\left( \frac{\partial }{\partial x}\right) ^{\alpha }f\left(
x+y^{\beta }\right) dy^{\beta }.
\end{gather*}%
We are using the notation $\left( \partial /\partial x\right) ^{\alpha
}=\left( \partial /\partial x_{1}\right) ^{\alpha _{1}}...\left( \partial
/\partial x_{d}\right) ^{\alpha _{d}}$ and $y^{\beta
}=\sum_{j=1}^{d}y_{j}^{\beta _{j}}e_{j}$, where $\left\{ e_{j}\right\}
_{j=1}^{d}$ is the canonical basis of $\mathbb{R}^{d}$ and $dy^{\beta
}=dy_{1}^{\beta _{1}}\ldots dy_{d}^{\beta _{d}}$.
\end{lemma}

\begin{proof}
In order to avoid a heavy notation and to make the proof transparent, we
consider only the case $d=2$ and we write $\left( x,y\right) $ and $\left(
m,n\right) $ in place of $x$ and $n$.
\begin{gather*}
\mathfrak{D}f(x,y)=\sum_{m\in \mathbb{Z}}\sum_{n\in \mathbb{Z}}\left(
2\delta \left( m\right) -2\pi im\right) \left( 2\delta \left( n\right) -2\pi
in\right) \widehat{f}\left( m,n\right) e^{2\pi i\left( mx+ny\right) } \\
=4\widehat{f}\left( 0,0\right) -2\sum_{m\in \mathbb{Z}}2\pi im\widehat{f}%
\left( m,0\right) e^{2\pi imx}-2\sum_{n\in \mathbb{Z}}2\pi in\widehat{f}%
\left( 0,n\right) e^{2\pi iny} \\
+\sum_{m\in \mathbb{Z}}\sum_{n\in \mathbb{Z}}\left( 2\pi im\right) \left(
2\pi in\right) \widehat{f}\left( m,n\right) e^{2\pi i\left( mx+ny\right) } \\
=4\int_{0}^{1}\int_{0}^{1}f(x,y)dxdy-2\int_{0}^{1}\frac{\partial f}{\partial
x}(x,y)dy-2\int_{0}^{1}\frac{\partial f}{\partial y}(x,y)dx+\frac{\partial
^{2}f}{\partial x\partial y}(x,y).
\end{gather*}
\end{proof}

The following lemma is nothing but a restatement of the theorem.

\begin{lemma}
\label{5}If $1\leq p,q\leq +\infty $ and $1/p+1/q=1$, then
\begin{gather*}
\left\vert \int_{\Omega }f(x)\overline{d\mu (x)}\right\vert \\
\leq \left( \int_{\left[ 0,1\right] ^{d}}\left\{ \int_{\mathbb{T}%
^{d}}\left\vert \sum_{n\in \mathbb{Z}^{d}}\mu \left( \left( x+n-I\left(
t\right) \right) \cap \Omega \right) \right\vert ^{q}dx\right\}
^{1/q}dt\right) \\
\times \left( \sum_{\alpha \in \left\{ 0,1\right\} ^{d}}2^{d-\left\vert
\alpha \right\vert }\left\{ \int_{\mathbb{T}^{d}}\left\vert \left( \frac{%
\partial }{\partial x}\right) ^{\alpha }f\left( x\right) \right\vert
^{p}dx\right\} ^{1/p}\right) .
\end{gather*}
\end{lemma}

\begin{proof}
We have to integrate a periodic function $f$ against a periodic measure $\mu
$ over an arbitrary non periodic Borel set $\Omega $ in $\mathbb{R}^{d}$. By
Lemma \ref{2} applied to the periodization $\nu $ of the measure $\chi
_{\Omega }\mu $, and by H\"{o}lder inequality,
\begin{gather*}
\left\vert \int_{\Omega }f(x)\overline{d\mu (x)}\right\vert =\left\vert
\int_{\mathbb{T}^{d}}f(x)\left( \sum_{n\in \mathbb{Z}^{d}}\chi _{\Omega
}\left( x+n\right) \right) \overline{d\mu (x)}\right\vert \\
=\left\vert \int_{\mathbb{T}^{d}}f(x)\overline{d\nu (x)}\right\vert \leq
\left\Vert \mathfrak{D}f\right\Vert _{L^{p}\left( \mathbb{T}^{d}\right)
}\left\Vert g\ast \nu \right\Vert _{L^{q}\left( \mathbb{T}^{d}\right) }.
\end{gather*}

The estimate for $\left\| \mathfrak{D}f\right\| _{L^{p}\left( \mathbb{T}%
^{d}\right) }$ follows from Lemma \ref{4},
\begin{equation*}
\left\{ \int_{\mathbb{T}^{d}}\left| \mathfrak{D}f\left( x\right) \right|
^{p}dx\right\} ^{1/p}\leq \sum_{\alpha \in \left\{ 0,1\right\}
^{d}}2^{d-\left| \alpha \right| }\left\{ \int_{\mathbb{T}^{d}}\left| \left(
\frac{\partial }{\partial x}\right) ^{\alpha }f\left( x\right) \right|
^{p}dx\right\} ^{1/p}.
\end{equation*}

The estimate for $\left\Vert g\ast \nu \right\Vert _{L^{q}\left( \mathbb{T}%
^{d}\right) }$ follows from Lemma \ref{3},
\begin{gather*}
\left\{ \int_{\mathbb{T}^{d}}\left\vert g\ast \nu (x)\right\vert
^{q}dx\right\} ^{1/q} \\
=\left\{ \int_{\mathbb{T}^{d}}\left\vert \int_{\mathbb{T}^{d}}\left(
\sum_{m\in \mathbb{Z}^{d}}\int_{\left[ 0,1\right] ^{d}}\chi _{I\left(
t\right) }\left( x-y+m\right) dt\right) \left( \sum_{n\in \mathbb{Z}%
^{d}}\chi _{\Omega }\left( y+n\right) \right) d\mu (y)\right\vert
^{q}dx\right\} ^{1/q} \\
=\left\{ \int_{\mathbb{T}^{d}}\left\vert \int_{\left[ 0,1\right] ^{d}}\left(
\sum_{n\in \mathbb{Z}^{d}}\int_{\mathbb{R}^{d}}\chi _{I\left( t\right)
}\left( x+n-z\right) \chi _{\Omega }\left( z\right) d\mu (z)\right)
dt\right\vert ^{q}dx\right\} ^{1/q} \\
=\left\{ \int_{\mathbb{T}^{d}}\left\vert \int_{\left[ 0,1\right] ^{d}}\left(
\sum_{n\in \mathbb{Z}^{d}}\mu \left( \left( x+n-I\left( t\right) \right)
\cap \Omega \right) \right) dt\right\vert ^{q}dx\right\} ^{1/q} \\
\leq \int_{\left[ 0,1\right] ^{d}}\left\{ \int_{\mathbb{T}^{d}}\left\vert
\sum_{n\in \mathbb{Z}^{d}}\mu \left( \left( x+n-I\left( t\right) \right)
\cap \Omega \right) \right\vert ^{q}dx\right\} ^{1/q}dt.
\end{gather*}
\end{proof}

The following are a few applications.

\begin{corollary}
\label{6}Let $\gamma =4$ if $d=2$, $\gamma =3/2$ if $d=3$, $\gamma =2/(d+1)$
if $d\geq 4$. Then there is a constant $c$ depending only on the dimension
such that for every $N\geq 2$ there exists a finite sequence of points $%
\{x_{j}\}_{j=1}^{N}$ in $\left[ 0,1\right] ^{d}$ with the following
property: For every convex set $\Omega $ contained in $\left[ 0,1\right]
^{d} $ and for every smooth function $f$ on $\mathbb{T}^{d}$,
\begin{gather*}
\left| N^{-1}\sum_{j=1}^{N}\left( f\chi _{\Omega }\right) \left(
x_{j}\right) -\int_{\Omega }f\left( x\right) dx\right| \\
\leq cN^{-2/(d+1)}\log ^{\gamma }\left( N\right) \sum_{\alpha \in \left\{
0,1\right\} ^{d}}\int_{\left[ 0,1\right] ^{d}}\left| \left( \frac{\partial }{%
\partial x}\right) ^{\alpha }f\left( x\right) \right| dx.
\end{gather*}
\end{corollary}

\begin{proof}
Apply the theorem to the measure $d\mu =N^{-1}\sum_{j=1}^{N}\delta
_{x_{j}}-dx$ with $p=1$ and $q=+\infty $. By results in \cite{B} and \cite{S}%
, there exist sequences of points with isotropic discrepancy, that is
discrepancy of convex sets,
\begin{equation*}
\sup\limits_{\text{convex }A\subseteq \left[ 0,1\right] ^{d}}\left\{
\left\vert N^{-1}\sum_{j=1}^{N}\chi _{A}\left( x_{j}\right) -\left\vert
A\right\vert \right\vert \right\} \leq cN^{-2/(d+1)}\log ^{\gamma }\left(
N\right) .
\end{equation*}
\end{proof}

For polyhedra there are sequences with smaller discrepancy.

\begin{corollary}
\label{7}Let $\gamma =1$ if $d=2$ and $\gamma =d$ if $d\geq 3$. Then for
every polyhedron $\Omega $ contained in $\left[ 0,1\right] ^{d}$ there is a
constant $c$ with the following property: For every $N\geq 2$ there exists a
finite sequence of points $\{x_{j}\}_{j=1}^{N}$ in $\left[ 0,1\right] ^{d}$
such that for every smooth function $f$ on $\mathbb{T}^{d}$,
\begin{gather*}
\left| N^{-1}\sum_{j=1}^{N}\left( f\chi _{\Omega }\right) \left(
x_{j}\right) -\int_{\Omega }f\left( x\right) dx\right| \\
\leq cN^{-1}\log ^{\gamma }\left( N\right) \sum_{\alpha \in \left\{
0,1\right\} ^{d}}\int_{\left[ 0,1\right] ^{d}}\left| \left( \frac{\partial }{%
\partial x}\right) ^{\alpha }f\left( x\right) \right| dx.
\end{gather*}
\end{corollary}

\begin{proof}
There exists a finite number of directions such that for every interval $I$
with edges parallel to the axes, the facets of the polyhedra $I\cap \Omega $
are perpendicular to these directions. Then one can apply Theorem 2.11 in
\cite{CGT} and deduce the existence of a finite sequence $\left\{
x_{j}\right\} _{j=1}^{N}$ with discrepancy
\begin{equation*}
\sup\limits_{I\subseteq \left[ 0,1\right] ^{d}}\left\{ \left\vert
N^{-1}\sum_{j=1}^{N}\chi _{I\cap \Omega }\left( x_{j}\right) -\left\vert
I\cap \Omega \right\vert \right\vert \right\} \leq cN^{-1}\log ^{d}\left(
N\right) .
\end{equation*}

When $d=2$, Theorem 1 in \cite{CT} gives the better estimate $cN^{-1}\,\log
\left( N\right) $.
\end{proof}

The following is another analog of Theorem \ref{1}, with a larger variation
but a smaller discrepancy.

\begin{theorem}
\label{8}If $0<a<1$ is an irrational number and if there exist $\delta >0$
and $\gamma \geq 2$ with the property that $\left\vert a-h/k\right\vert \geq
\delta k^{-\gamma }$ for every rational $h/k$, then there exists a constant $%
c>0$, which depends explicitly on $\delta $, $\gamma $, $d$, with the
following property: If $A=\left[ -a/2,+a/2\right] ^{d}$ is the cube centered
at the origin with side length $a$, if $\Omega $ is a Borel set in $\mathbb{R%
}^{d}$, if $\mu $ is a periodic Borel measure, and if $f$ is a periodic
smooth function, then
\begin{gather*}
\left\vert \int_{\Omega }f\left( x\right) \overline{d\mu (x)}\right\vert \\
\leq c\left\{ \int_{\mathbb{T}^{d}}\left\vert \sum_{n\in \mathbb{Z}^{d}}\mu
\left( \left( x+n-A\right) \cap \Omega \right) \right\vert ^{2}dx\right\}
^{1/2} \\
\times \left\{ \sum_{n\in \mathbb{Z}^{d}}\left( \prod_{k=1}^{d}\left(
1+\left\vert n_{k}\right\vert \right) ^{2\gamma }\right) \left\vert \widehat{%
f}\left( n\right) \right\vert ^{2}\right\} ^{1/2}.
\end{gather*}
\end{theorem}

\begin{proof}
The periodization of $\chi _{A}$ has Fourier expansion
\begin{equation*}
g(x)=\sum_{n\in \mathbb{Z}^{d}}\left( \prod_{k=1}^{d}\frac{\sin \left( \pi
an_{k}\right) }{\pi n_{k}}\right) e^{2\pi in\cdot x}.
\end{equation*}

By Lemma \ref{2} applied to the periodization $\nu $ of the measure $\chi
_{\Omega }\mu $, with $p=q=2$,
\begin{gather*}
\left\vert \int_{\Omega }f(x)\overline{d\mu (x)}\right\vert \leq \left\{
\int_{\mathbb{T}^{d}}\left\vert g\ast \nu (x)\right\vert ^{2}dx\right\}
^{1/2} \\
\times \left\{ \int_{\mathbb{T}^{d}}\left\vert \sum_{n\in \mathbb{Z}%
^{d}}\left( \prod_{k=1}^{d}\frac{\sin \left( \pi an_{k}\right) }{\pi n_{k}}%
\right) ^{-1}\widehat{f}\left( n\right) e^{2\pi in\cdot x}\right\vert
^{2}dx\right\} ^{1/2}.
\end{gather*}

As in the proof of Lemma \ref{5},
\begin{equation*}
\left\{ \int_{\mathbb{T}^{d}}\left\vert g\ast \nu (x)\right\vert
^{2}dx\right\} ^{1/2}=\left\{ \int_{\mathbb{T}^{d}}\left\vert \sum_{n\in
\mathbb{Z}^{d}}\mu \left( \left( x+n-A\right) \cap \Omega \right)
\right\vert ^{2}dx\right\} ^{1/2}.
\end{equation*}

By the assumptions, $\left| \sin \left( \pi an_{k}\right) \right| \geq
c\left| n_{k}\right| ^{1-\gamma }$ when $n_{k}\neq 0$. Then
\begin{equation*}
\left| \prod_{k=1}^{d}\frac{\sin \left( \pi an_{k}\right) }{\pi n_{k}}%
\right| ^{-1}\leq c\prod_{k=1}^{d}\left( 1+\left| n_{k}\right| \right)
^{\gamma }.
\end{equation*}

Hence, by Parseval equality,
\begin{gather*}
\left\{ \int_{\mathbb{T}^{d}}\left| \sum_{n\in \mathbb{Z}^{d}}\left(
\prod_{k=1}^{d}\frac{\sin \left( \pi an_{k}\right) }{\pi n_{k}}\right) ^{-1}%
\widehat{f}\left( n\right) e^{2\pi in\cdot x}\right| ^{2}dx\right\} ^{1/2} \\
=\left\{ \sum_{n\in \mathbb{Z}^{d}}\left| \left( \prod_{k=1}^{d}\frac{\sin
\left( \pi an_{k}\right) }{\pi n_{k}}\right) ^{-1}\widehat{f}\left( n\right)
\right| ^{2}\right\} ^{1/2} \\
\leq c\left\{ \sum_{n\in \mathbb{Z}^{d}}\left( \prod_{k=1}^{d}\left(
1+\left| n_{k}\right| \right) ^{2\gamma }\right) \left| \widehat{f}\left(
n\right) \right| ^{2}\right\} ^{1/2}.
\end{gather*}
\end{proof}

In particular, if $a$ is a quadratic irrational one can take $\gamma =2$ and
$\delta $ can be made explicit, and the variation of the function can be
controlled by square norms of derivatives up to the order $2d$. This
variation is larger than the one in the proof of Theorem \ref{1}, which is
controlled by derivatives of order $d$. On the other hand, the discrepancy
associated to the family of all intervals in Theorem \ref{1} is larger than
the discrepancy associated to the translated of a single cube in Corollary %
\ref{8}. Finally, there is an analog of the above theorem with balls instead
of cubes. The zeros of Fourier transforms of characteristic functions of
cubes play a crucial role in the above corollary. The Fourier transforms of
balls can be expressed in terms of Bessel functions. The following lemma is
about the zeros of Bessel functions.

\begin{lemma}
\label{9}If $J_{\alpha }$ is the Bessel function of first kind of order $%
\alpha \geq -1/2$, if $\beta >5/4$, and if $0<a<b$, then there exists $c>0$
and $a<r<b$ with the property that for every positive integer $k$,
\begin{equation*}
\left| J_{\alpha }\left( r\sqrt{k}\right) \right| \geq ck^{-\beta }.
\end{equation*}
\end{lemma}

\begin{proof}
The zeros of $J_{\alpha }\left( t\right) $ are simple, with the possible
exception of $t=0$. If $\varepsilon <1$ then $\left| J_{\alpha }\left(
t\right) \right| ^{-\varepsilon }$ is locally integrable in $t>0$ and, by
the asymptotic expansion of Bessel functions,
\begin{gather*}
\int_{a}^{b}\left| \sqrt{r\sqrt{k}}J_{\alpha }\left( r\sqrt{k}\right)
\right| ^{-\varepsilon }dr=k^{-1/2}\int_{a\sqrt{k}}^{b\sqrt{k}}\left| \sqrt{t%
}J_{\alpha }\left( t\right) \right| ^{-\varepsilon }dt \\
=k^{-1/2}\int_{a\sqrt{k}}^{b\sqrt{k}}\left| \sqrt{2/\pi }\cos \left(
t-\alpha \pi /2-\pi /4\right) +O\left( t^{-1}\right) \right| ^{-\varepsilon
}dt\leq c.
\end{gather*}

Hence, if $\varepsilon <1$ and $\eta >1+\varepsilon /4$,
\begin{gather*}
\int_{a}^{b}\left( \sum_{k=1}^{+\infty }k^{-\eta }\left| J_{\alpha }\left( r%
\sqrt{k}\right) \right| ^{-\varepsilon }\right) dr \\
\leq b^{\varepsilon /2}\sum_{k=1}^{+\infty }k^{\varepsilon /4-\eta }\left(
\int_{a}^{b}\left| \sqrt{r\sqrt{k}}J_{\alpha }\left( r\sqrt{k}\right)
\right| ^{-\varepsilon }dr\right) <+\infty .
\end{gather*}

Since the series $\sum_{k=1}^{+\infty }k^{-\eta }\left\vert J_{\alpha
}\left( r\sqrt{k}\right) \right\vert ^{-\varepsilon }$ converges for almost
every $r$, for almost every $r$ there exists $c>0$ such that
\begin{equation*}
\left\vert J_{\alpha }\left( r\sqrt{k}\right) \right\vert \geq ck^{-\eta
/\varepsilon }.
\end{equation*}

Finally observe that if the interval $r\sqrt{k}\leq t\leq r\sqrt{k+1}$
contains a zero of $J_{\alpha }\left( t\right) $ then $\left\vert J_{\alpha
}\left( r\sqrt{k}\right) \right\vert \leq ck^{-3/4}$. Hence the thesis does
not hold with $\beta <3/4$.
\end{proof}

\begin{theorem}
\label{10}For every $0<a<b$ and $\gamma >d/2+5/4$ there exist a constant $%
c>0 $ and a radius $a<r<b$ with the following properties: If $B=\left\{
\left\vert x\right\vert \leq r\right\} $ is the ball centered at the origin
with radius $r$, if $\Omega $ is a Borel set in $\mathbb{R}^{d}$, if $\mu $
is a periodic Borel measure, and if $f$ is a periodic smooth function, then
\begin{gather*}
\left\vert \int_{\Omega }f\left( x\right) \overline{d\mu (x)}\right\vert \\
\leq c\left\{ \int_{\mathbb{T}^{d}}\left\vert \sum_{n\in \mathbb{Z}^{d}}\mu
\left( \left( x+n-B\right) \cap \Omega \right) \right\vert ^{2}dx\right\}
^{1/2}\left\{ \sum_{n\in \mathbb{Z}^{d}}\left( 1+\left\vert n\right\vert
^{2}\right) ^{\gamma }\left\vert \widehat{f}\left( n\right) \right\vert
^{2}\right\} ^{1/2}.
\end{gather*}
\end{theorem}

\begin{proof}
The Fourier transform of the characteristic function of a ball is a Bessel
function, and the periodization of $\chi _{B}$ has Fourier expansion
\begin{equation*}
g(x)=\sum_{n\in \mathbb{Z}^{d}}r^{d/2}\left| n\right| ^{-d/2}J_{d/2}\left(
2\pi r\left| n\right| \right) e^{2\pi in\cdot x}.
\end{equation*}

By Lemma \ref{2} applied to the periodization $\nu $ of the measure $\chi
_{\Omega }\mu $, with $p=q=2$,
\begin{gather*}
\left\vert \int_{\Omega }f(x)\overline{d\mu (x)}\right\vert \leq \left\{
\int_{\mathbb{T}^{d}}\left\vert g\ast \nu (x)\right\vert ^{2}dx\right\}
^{1/2} \\
\times \left\{ \int_{\mathbb{T}^{d}}\left\vert \sum_{n\in \mathbb{Z}%
^{d}}\left( r^{d/2}\left\vert n\right\vert ^{-d/2}J_{d/2}\left( 2\pi
r\left\vert n\right\vert \right) \right) ^{-1}\widehat{f}\left( n\right)
e^{2\pi in\cdot x}\right\vert ^{2}dx\right\} ^{1/2}.
\end{gather*}

As in the proof of Lemma \ref{5},
\begin{equation*}
\left\{ \int_{\mathbb{T}^{d}}\left\vert g\ast \nu (x)\right\vert
^{2}dx\right\} ^{1/2}=\left\{ \int_{\mathbb{T}^{d}}\left\vert \sum_{n\in
\mathbb{Z}^{d}}\mu \left( \left( x+n-B\right) \cap \Omega \right)
\right\vert ^{2}dx\right\} ^{1/2}.
\end{equation*}

Moreover, by Parseval equality and Lemma \ref{9}, one can choose $r$ so that
\begin{gather*}
\left\{ \int_{\mathbb{T}^{d}}\left| \sum_{n\in \mathbb{Z}^{d}}\left(
r^{d/2}\left| n\right| ^{-d/2}J_{d/2}\left( 2\pi r\left| n\right| \right)
\right) ^{-1}\widehat{f}\left( n\right) e^{2\pi in\cdot x}\right|
^{2}dx\right\} ^{1/2} \\
=\left\{ \sum_{n\in \mathbb{Z}^{d}}\left| \left( r^{d/2}\left| n\right|
^{-d/2}J_{d/2}\left( 2\pi r\left| n\right| \right) \right) ^{-1}\widehat{f}%
\left( n\right) \right| ^{2}\right\} ^{1/2} \\
\leq c\left\{ \sum_{n\in \mathbb{Z}^{d}}\left( 1+\left| n\right| ^{2}\right)
^{\gamma }\left| \widehat{f}\left( n\right) \right| ^{2}\right\} ^{1/2}.
\end{gather*}

Finally observe that less than $d/2$ square integrable derivatives are not
enough to guarantee the boundedness of a functions. Hence the assumption $%
\gamma >d/2+5/4$ is not too far to be best possible.
\end{proof}

\section{Koksma Hlawka inequalities on manifolds}

The results in the previous section are of local nature and with a change of
variables can be easily transferred from cubes to compact manifolds. Let $%
\mathcal{M}$ be a smooth compact $d$ dimensional manifold with a normalized
measure $dx$. Fix a family of local charts $\left\{ \varphi _{k}\right\}
_{k=1}^{K}$, $\varphi _{k}:\left[ 0,1\right] ^{d}\rightarrow \mathcal{M}$,
and a smooth partition of unity $\left\{ \psi _{k}\right\} _{k=1}^{K}$
subordinate to these charts. The Sobolev spaces $W^{n,p}\left( \mathcal{M}%
\right) $ can be defined by the norms
\begin{equation*}
\left\Vert f\right\Vert _{W^{n,p}\left( \mathcal{M}\right) }=\sum_{1\leq
k\leq K}\sum_{\left\vert \alpha \right\vert \leq n}\left\{ \int_{\left[ 0,1%
\right] ^{d}}\left\vert \frac{\partial ^{\alpha }}{\partial x^{\alpha }}%
\left( \psi _{k}\left( \varphi _{k}(x)\right) f\left( \varphi _{k}(x)\right)
\right) \right\vert ^{p}dx\right\} ^{1/p}.
\end{equation*}

One can define an interval in $\mathcal{M}$ as the image under a local chart
of an interval in $\left[ 0,1\right] ^{d}$, say $U=\varphi _{k}\left(
I\right) $. The discrepancy of a finite Borel measure $\mu $ on $\mathcal{M}$
with respect to the collection $A$ of all intervals in $\mathcal{M}$ is
\begin{equation*}
\mathcal{D}\left( \mu \right) =\sup\limits_{U\in A}\left| \int_{U}d\mu
(y)\right| .
\end{equation*}

\begin{theorem}
\label{11}There exists a constant $c>0$, which depends on the local charts
but not on the function $f$ or the measure $\mu $, such that
\begin{equation*}
\left\vert \int_{\mathcal{M}}f(y)\overline{d\mu \left( y\right) }\right\vert
\leq c\mathcal{D}\left( \mu \right) \left\Vert f\right\Vert _{W^{d,1}\left(
\mathcal{M}\right) }.
\end{equation*}
\end{theorem}

\begin{proof}
It suffices to prove the theorem for functions with support in the image of
a single local chart $\varphi :\left[ 0,1\right] ^{d}\rightarrow \mathcal{M}$%
. If the measure $\nu $ is the pull back on $\left[ 0,1\right] ^{d}$ of the
measure $\mu $ on $\mathcal{M}$ then, by Theorem \ref{1},
\begin{gather*}
\left\vert \int_{\mathcal{M}}f(y)\overline{d\mu \left( y\right) }\right\vert
=\left\vert \int_{\left[ 0,1\right] ^{d}}f\left( \varphi \left( x\right)
\right) \overline{d\nu \left( x\right) }\right\vert \\
\leq 2^{d}\left\{ \sup\limits_{I\subseteq \left[ 0,1\right] ^{d}}\left\vert
\int_{\left[ 0,1\right] ^{d}}\chi _{I}\left( x\right) d\nu \left( x\right)
\right\vert \right\} \\
\times \left\{ \sum_{\alpha \in \left\{ 0,1\right\} ^{d}}2^{d-\left\vert
\alpha \right\vert }\int_{\left[ 0,1\right] ^{d}}\left\vert \left( \frac{%
\partial }{\partial x}\right) ^{\alpha }f\left( \varphi \left( x\right)
\right) \right\vert dx\right\} .
\end{gather*}

The first factor is dominated by the discrepancy,
\begin{equation*}
\sup\limits_{I\subseteq \left[ 0,1\right] ^{d}}\left\{ \left| \int_{\left[
0,1\right] ^{d}}\chi _{I}\left( x\right) d\nu \left( x\right) \right|
\right\} \leq \sup\limits_{U\in A}\left\{ \left| \int_{\mathcal{M}}\chi
_{U}(y)d\mu \left( y\right) \right| \right\} .
\end{equation*}

The second factor is dominated by the Sobolev norm,
\begin{equation*}
\sum_{\alpha \in \left\{ 0,1\right\} ^{d}}2^{d-\left| \alpha \right| }\int_{
\left[ 0,1\right] ^{d}}\left| \left( \frac{\partial }{\partial x}\right)
^{\alpha }f\left( \varphi \left( x\right) \right) \right| dx\leq c\left\|
f\right\| _{W^{d,1}\left( \mathcal{M}\right) }.
\end{equation*}
\end{proof}

For example, when the manifold is a 2 dimensional sphere and the local
charts are central projection from a tangent plane to the sphere, the images
of rectangles on the tangent plane are geodesic quadrilaterals on the
sphere. Since a quadrilateral is union of two triangles, everything can be
controlled by the discrepancy with respect to geodesic triangles. Finally,
as in Theorem \ref{10}, one can consider a Koksma Hlawka inequality on the
sphere, with spherical cap discrepancy. The zonal polynomials $Z_{n}\left(
x\cdot y\right) $ on the sphere $\mathcal{S}=\left\{ x\in \mathbb{R}%
^{3}:\;\left| x\right| =1\right\} $ are the reproducing kernels of the
spaces of harmonic polynomials of degree $n$. If $Q_{n}(x)$ is a harmonic
polynomial of degree $n$, then
\begin{equation*}
Q_{n}(x)=\int_{\mathcal{S}}Z_{n}\left( x\cdot y\right) Q_{n}(y)dy.
\end{equation*}

Every distribution $f$ on the sphere has a spherical harmonic expansion
\begin{equation*}
f\left( x\right) =\sum_{n=0}^{+\infty }\int_{\mathcal{S}}Z_{n}\left( x\cdot
y\right) f\left( y\right) dy=\sum_{n=0}^{+\infty }\widehat{f}\left(
n,x\right) .
\end{equation*}

The series converges in the topology of distributions. If $f$ is square
integrable, the series converges in the square norm, and if it is smooth, it
also converges absolutely and uniformly. The following is a spherical analog
of Lemma \ref{2}.

\begin{lemma}
\label{12} Let $f$ be an integrable function and $\mu $ a finite measure on
the sphere. Also let $\varphi (n)$ be a non vanishing complex sequence on $%
\mathbb{N}$, and assume that both $\varphi (n)$ and $1/\varphi (n)$ have
tempered growth. Define
\begin{gather*}
\mathfrak{D}f\left( x\right) =\sum_{n=0}^{+\infty }\varphi \left( n\right)
\widehat{f}\left( n,x\right) , \\
g\left( x\cdot y\right) =\sum_{n=0}^{+\infty }\overline{\varphi \left(
n\right) ^{-1}}Z_{n}\left( x\cdot y\right) .
\end{gather*}
Then
\begin{equation*}
\left| \int_{\mathcal{S}}f(x)\overline{d\mu (x)}\right| \leq \left\{ \int_{%
\mathcal{S}}\left| \int_{\mathcal{S}}g\left( x\cdot y\right) d\mu (y)\right|
^{2}dx\right\} ^{1/2}\left\{ \int_{\mathcal{S}}\left| \mathfrak{D}f\left(
x\right) \right| ^{2}dx\right\} ^{1/2}.
\end{equation*}
\end{lemma}

\begin{proof}
By the spherical harmonic expansions of $f$ and $\mu $,
\begin{gather*}
\left| \int_{\mathcal{S}}f(x)\overline{d\mu (x)}\right| \\
=\left| \sum_{n=0}^{+\infty }\int_{\mathcal{S}}\widehat{f}\left( n,x\right)
\overline{\widehat{\mu }\left( n,x\right) }dx\right| \\
\leq \left\{ \int_{\mathcal{S}}\left| \sum_{n=0}^{+\infty }\varphi \left(
n\right) \widehat{f}\left( n,x\right) \right| ^{2}dx\right\} ^{1/2} \\
\times \left\{ \int_{\mathcal{S}}\left| \int_{\mathcal{S}}\left(
\sum_{n=0}^{+\infty }\overline{\varphi \left( n\right) ^{-1}}Z_{n}\left(
x\cdot y\right) \right) d\mu (y)\right| ^{2}dx\right\} ^{1/2}.
\end{gather*}
\end{proof}

In what follows two specific examples of sequences $\varphi $ and functions $%
g$ are considered. The following is an analog of Lemma \ref{9}, with
Legendre polynomials in place of Bessel functions.

\begin{lemma}
\label{13} (1) Let
\begin{equation*}
\chi _{\left\{ x\cdot y\geq \cos \left( \vartheta \right) \right\} }\left(
x\cdot y\right) =\sum_{n=0}^{+\infty }\overline{\varphi (n)^{-1}}Z_{n}\left(
x\cdot y\right)
\end{equation*}
be the spherical harmonic expansion of the characteristic function of the
spherical cap $\left\{ x\cdot y\geq \cos \left( \vartheta \right) \right\} $
on the two dimensional sphere. Then for every $\gamma >5/2$ and for almost
every $0<\vartheta <\pi $ there exists positive constants $c_{1}$ and $c_{2}$
such that for every positive integer $n$,
\begin{equation*}
c_{1}n^{3/2}\leq \left| \varphi (n)\right| \leq c_{2}n^{\gamma }.
\end{equation*}

(2) Let
\begin{equation*}
\chi _{\left\{ x\cdot y\geq \cos \left( \vartheta \right) \right\} }\left(
x\cdot y\right) +i\chi _{\left\{ x\cdot y\geq \cos \left( 2\vartheta \right)
\right\} }\left( x\cdot y\right) =\sum_{n=0}^{+\infty }\overline{\varphi
(n)^{-1}}Z_{n}\left( x\cdot y\right) .
\end{equation*}%
Then for almost every $0<\vartheta <\pi /2$ there exist positive constants $%
c_{1}$ and $c_{2}$ such that for every positive integer $n$,
\begin{equation*}
c_{1}n^{3/2}\leq \left\vert \varphi (n)\right\vert \leq c_{2}n^{3/2}.
\end{equation*}
\end{lemma}

\begin{proof}
The zonal polynomials on the two dimensional sphere are multiple of Legendre
polynomials,
\begin{eqnarray*}
Z_{n}\left( x\cdot y\right) &=&\left( 2n+1\right) P_{n}\left( x\cdot
y\right) , \\
P_{n}\left( z\right) &=&\dfrac{d^{n}}{dz^{n}}\dfrac{\left( z^{2}-1\right)
^{n}}{2^{n}n!}.
\end{eqnarray*}

The characteristic function of the spherical cap $\left\{ x\cdot y\geq \cos
\left( \vartheta \right) \right\} $ has the expansion
\begin{gather*}
\chi _{\left\{ x\cdot y\geq \cos \left( \vartheta \right) \right\} }\left(
x\cdot y\right) \\
=\sum_{n=0}^{+\infty }\left( \left( n+1/2\right) \int_{\cos \left( \vartheta
\right) }^{1}P_{n}\left( z\right) dz\right) P_{n}\left( x\cdot y\right) \\
=\dfrac{1-\cos \left( \vartheta \right) }{2}+\sum_{n=1}^{+\infty }\dfrac{%
P_{n-1}\left( \cos \left( \vartheta \right) \right) -P_{n+1}\left( \cos
\left( \vartheta \right) \right) }{2}P_{n}\left( x\cdot y\right) \\
=\dfrac{1-\cos \left( \vartheta \right) }{2}Z_{0}\left( x\cdot y\right)
+\sum_{n=1}^{+\infty }\dfrac{P_{n-1}\left( \cos \left( \vartheta \right)
\right) -P_{n+1}\left( \cos \left( \vartheta \right) \right) }{2\left(
2n+1\right) }Z_{n}\left( x\cdot y\right) .
\end{gather*}

This follows from the identities
\begin{equation*}
Z_{n}\left( x\cdot y\right) =\left( 2n+1\right) P_{n}\left( z\right)
=P_{n+1}^{^{\prime }}\left( z\right) -P_{n-1}^{\prime }\left( z\right) .
\end{equation*}

The Legendre polynomials in $0<a<\vartheta <b<\pi $ have the asymptotic
expansion
\begin{equation*}
P_{n}\left( \cos \left( \vartheta \right) \right) =\sqrt{\dfrac{2}{\pi n\sin
\left( \vartheta \right) }}\cos \left( \left( n+1/2\right) \vartheta +\pi
/4\right) +O\left( n^{-3/2}\right) .
\end{equation*}

Hence
\begin{gather*}
P_{n-1}\left( \cos \left( \vartheta \right) \right) -P_{n+1}\left( \cos
\left( \vartheta \right) \right) \\
=\sqrt{\dfrac{2}{\pi \left( n-1\right) \sin \left( \vartheta \right) }}\cos
\left( \left( n-1/2\right) \vartheta +\pi /4\right) \\
-\sqrt{\dfrac{2}{\pi \left( n+1\right) \sin \left( \vartheta \right) }}\cos
\left( \left( n+3/2\right) \vartheta +\pi /4\right) +O\left( n^{-3/2}\right)
\\
=\sqrt{\dfrac{2}{\pi n\sin \left( \vartheta \right) }}\left( \cos \left(
\left( n-1/2\right) \vartheta +\pi /4\right) -\cos \left( \left(
n+3/2\right) \vartheta +\pi /4\right) \right) +O\left( n^{-3/2}\right) \\
=\sqrt{\dfrac{2}{\pi n\sin \left( \vartheta \right) }}2\sin \left( \vartheta
\right) \sin \left( \left( n+1/2\right) \vartheta +\pi /4\right) +O\left(
n^{-3/2}\right) .
\end{gather*}

It follows from these estimates that $\left| \varphi (n)\right| \geq
cn^{3/2} $. In order to prove a reverse inequality, observe that the above
polynomial vanishes only when $\left( n+1/2\right) \vartheta +\pi /4$ is
close to a multiple of $\pi $, but at these points the derivative is large,
\begin{gather*}
\frac{d}{d\vartheta }\left( P_{n-1}\left( \cos \left( \vartheta \right)
\right) -P_{n+1}\left( \cos \left( \vartheta \right) \right) \right) \\
=\left( 2n+1\right) \sin \left( \vartheta \right) P_{n}\left( \cos \left(
\vartheta \right) \right) \\
=\left( 2n+1\right) \sqrt{\dfrac{2\sin \left( \vartheta \right) }{\pi n}}%
\cos \left( \left( n+1/2\right) \vartheta +\pi /4\right) +O\left(
n^{-1/2}\right) .
\end{gather*}

In particular, if $n$ is large, say $n\geq N$, and $0<a<\vartheta <b<\pi $
then the zeros are simple. This implies that in (1),
\begin{gather*}
\varphi (n)=\left( \dfrac{P_{n-1}\left( \cos \left( \vartheta \right)
\right) -P_{n+1}\left( \cos \left( \vartheta \right) \right) }{2\left(
2n+1\right) }\right) ^{-1} \\
=\dfrac{\sqrt{\pi n}\left( 2n+1\right) }{\sqrt{2\sin \left( \vartheta
\right) }\left( \sin \left( \left( n+1/2\right) \vartheta +\pi /4\right)
+O\left( 1/n\right) \right) }.
\end{gather*}

Finally, as in the proof of Lemma \ref{9}, if $\varepsilon <1$ and $\eta
>1+3\varepsilon /2$ then
\begin{equation*}
\int_{a}^{b}\left( \sum_{n=N}^{+\infty }n^{-\eta }\left| \varphi (n)\right|
^{\varepsilon }\right) d\vartheta <+\infty .
\end{equation*}

If the series $\sum_{n=1}^{+\infty }n^{-\eta }\left| \varphi (n)\right|
^{\varepsilon }$ converges for almost every $\vartheta $, then for almost
every $\vartheta $ there exists $c>0$ such that
\begin{equation*}
\left| \varphi (n)\right| \leq cn^{\eta /\varepsilon }.
\end{equation*}

This proves (1). The proof of (2) is a bit different. Let
\begin{equation*}
\chi _{\left\{ x\cdot y\geq \cos \left( \vartheta \right) \right\} }\left(
x\cdot y\right) +i\chi _{\left\{ x\cdot y\geq \cos \left( 2\vartheta \right)
\right\} }\left( x\cdot y\right) =\sum_{n=0}^{+\infty }\overline{\varphi
(n)^{-1}}Z_{n}\left( x\cdot y\right) .
\end{equation*}

Then, if $n=1,2,3,...$,
\begin{gather*}
\left| \varphi (n)\right| \\
=\left| \dfrac{P_{n-1}\left( \cos \left( \vartheta \right) \right)
-P_{n+1}\left( \cos \left( \vartheta \right) \right) }{2\left( 2n+1\right) }%
+i\dfrac{P_{n-1}\left( \cos \left( 2\vartheta \right) \right) -P_{n+1}\left(
\cos \left( 2\vartheta \right) \right) }{2\left( 2n+1\right) }\right| ^{-1}
\\
=\left( \left| \dfrac{P_{n-1}\left( \cos \left( \vartheta \right) \right)
-P_{n+1}\left( \cos \left( \vartheta \right) \right) }{2\left( 2n+1\right) }%
\right| ^{2}+\left| \dfrac{P_{n-1}\left( \cos \left( 2\vartheta \right)
\right) -P_{n+1}\left( \cos \left( 2\vartheta \right) \right) }{2\left(
2n+1\right) }\right| ^{2}\right) ^{-1/2} \\
=\sqrt{\pi n/2}\left( 2n+1\right) \left( \left| \sin \left( \vartheta
\right) \right| \sin ^{2}\left( \left( n+1/2\right) \vartheta +\pi /4\right)
\right. \\
\left. +\left| \sin \left( 2\vartheta \right) \right| \sin ^{2}\left( \left(
2n+1\right) \vartheta +\pi /4\right) +O\left( 1/n\right) \right) ^{-1/2}.
\end{gather*}

If $0<a<\vartheta <b<\pi /2$ and if $n$ is large, then
\begin{gather*}
\left| \sin \left( \vartheta \right) \right| \sin ^{2}\left( \left(
n+1/2\right) \vartheta +\pi /4\right) +\left| \sin \left( 2\vartheta \right)
\right| \sin ^{2}\left( \left( 2n+1\right) \vartheta +\pi /4\right) +O\left(
1/n\right) \\
\geq c\left( \sin ^{2}\left( \omega +\pi /4\right) +\sin ^{2}\left( 2\omega
+\pi /4\right) \right) +O\left( 1/n\right) \\
\geq c+O\left( 1/n\right) \geq c>0.
\end{gather*}

In particular, there exists $N$ such that for every $n\geq N$ and every $%
0<a<\vartheta <b<\pi /2$, one has $\left| \varphi (n)\right| \leq cn^{3/2}$.
Moreover, the equations $\varphi (n)=0$ for some $n<N$ have a finite number
of solutions $0<a<\vartheta <b<\pi /2$. Hence, if $\vartheta $ is not one of
these solutions, then it satisfies (2).
\end{proof}

\begin{theorem}
\label{14} (1) For every $\gamma >5/2$ and almost every $0<\vartheta <\pi $
there exists a constant $c>0$ with the following property: If $B\left(
x,\vartheta \right) =\left\{ x\cdot y\geq \cos \left( \vartheta \right)
\right\} $ are the spherical caps with center $x$ and radius $\vartheta $,
if $\Omega $ is a Borel set, if $\mu $ is a Borel measure, and if $f$ is a
smooth function in $\mathcal{S}$, then
\begin{gather*}
\left\vert \int_{\Omega }f\left( x\right) \overline{d\mu (x)}\right\vert \\
\leq c\left\{ \int_{\mathcal{S}}\left\vert \mu \left( B\left( x,\vartheta
\right) \cap \Omega \right) \right\vert ^{2}dx\right\} ^{1/2}\left\{
\sum_{n=0}^{+\infty }\left( 1+n^{2}\right) ^{\gamma }\int_{\mathcal{S}%
}\left\vert \widehat{f}\left( n,x\right) \right\vert ^{2}dx\right\} ^{1/2}.
\end{gather*}

(2) For almost every radius $0<\vartheta <\pi /2$ there exists a constant $%
c>0$ with the property that
\begin{gather*}
\left\vert \int_{\Omega }f\left( x\right) \overline{d\mu (x)}\right\vert \\
\leq c\left\{ \sum_{n=0}^{+\infty }\left( 1+n^{2}\right) ^{3/2}\int_{%
\mathcal{S}}\left\vert \widehat{f}\left( n,x\right) \right\vert
^{2}dx\right\} ^{1/2} \\
\times \left( \left\{ \int_{\mathcal{S}}\left\vert \mu \left( B\left(
x,\vartheta \right) \cap \Omega \right) \right\vert ^{2}dx\right\}
^{1/2}+\left\{ \int_{\mathcal{S}}\left\vert \mu \left( B\left( x,2\vartheta
\right) \cap \Omega \right) \right\vert ^{2}dx\right\} ^{1/2}\right) .
\end{gather*}
\end{theorem}

\begin{proof}
If $g\left( x\cdot y\right) =\chi _{\left\{ x\cdot y\geq \cos \left(
\vartheta \right) \right\} }\left( x\cdot y\right) $ and if $\nu $ is the
restriction of the measure $\mu $ to the set $\Omega $, then
\begin{equation*}
\left\{ \int_{\mathcal{S}}\left\vert \int_{\mathcal{S}}g\left( x\cdot
y\right) d\nu (y)\right\vert ^{2}dx\right\} ^{1/2}=\left\{ \int_{\mathcal{S}%
}\left\vert \mu \left( B\left( x,\vartheta \right) \cap \Omega \right)
\right\vert ^{2}dx\right\} ^{1/2}.
\end{equation*}

Then (1) follows from Lemma \ref{12} and Lemma \ref{13} (1). If $g\left(
x\cdot y\right) =\chi _{\left\{ x\cdot y\geq \cos \left( \vartheta \right)
\right\} }\left( x\cdot y\right) +i\chi _{\left\{ x\cdot y\geq \cos \left(
2\vartheta \right) \right\} }\left( x\cdot y\right) $ and if $\nu $ is the
restriction of the measure $\mu $ to the set $\Omega $, then
\begin{gather*}
\left\{ \int_{\mathcal{S}}\left\vert \int_{\mathcal{S}}g\left( x\cdot
y\right) d\nu (y)\right\vert ^{2}dx\right\} ^{1/2} \\
\leq \left\{ \int_{\mathcal{S}}\left\vert \int_{\mathcal{S}}\chi _{\left\{
x\cdot y\geq \cos \left( \vartheta \right) \right\} }\left( x\cdot y\right)
d\nu (y)\right\vert ^{2}dx\right\} ^{1/2} \\
+\left\{ \int_{\mathcal{S}}\left\vert \int_{\mathcal{S}}\chi _{\left\{
x\cdot y\geq \cos \left( 2\vartheta \right) \right\} }\left( x\cdot y\right)
d\nu (y)\right\vert ^{2}dx\right\} ^{1/2} \\
=\left\{ \int_{\mathcal{S}}\left\vert \mu \left( B\left( x,\vartheta \right)
\cap \Omega \right) \right\vert ^{2}dx\right\} ^{1/2}+\left\{ \int_{\mathcal{%
S}}\left\vert \mu \left( B\left( x,2\vartheta \right) \cap \Omega \right)
\right\vert ^{2}dx\right\} ^{1/2}.
\end{gather*}

Then, as before, (2) follows from Lemma \ref{12} and Lemma \ref{13} (2).
\end{proof}

Observe that the indices $\gamma >9/4$ in Theorem \ref{10} with $d=2$ and $%
\gamma >5/2$ in Theorem \ref{14} (1) are different. Anyhow, it is likely
that both indices are not best possible. Finally, an analog of Theorem \ref%
{14} (1) holds on spheres of dimension $d>2$ with $\gamma >\left( d+3\right)
/2$.

\section{An application}

In order to test the quality of the above results, we reconsider an example
in \cite{Har10}. Let
\begin{equation*}
f\left( x_{1},x_{2},\ldots ,x_{d}\right) =\frac{1}{x_{1}x_{2}\cdots
x_{d}\left( 1-x_{1}-x_{2}-\ldots -x_{d}\right) }.
\end{equation*}

Also, for $\varepsilon >0$ small, let $\Sigma $ be the simplex
\begin{equation*}
\Sigma =\left\{ \left( x_{1},\ldots ,x_{d}\right) \in \mathbb{R}%
^{d}:x_{1}\geq \ldots \geq x_{d}\geq \varepsilon ,\;1-x_{1}-\ldots
-x_{d}\geq \varepsilon \right\} .
\end{equation*}

One can show that
\begin{equation*}
\sum_{\left\vert \alpha \right\vert \leq d}\int_{\Sigma }\left\vert \left(
\frac{\partial }{\partial x}\right) ^{\alpha }f\left( x\right) \right\vert
dx\leq c\varepsilon ^{-d}.
\end{equation*}

By Corollary \ref{6} there exists a finite sequence $\left\{ x_{j}\right\}
_{j=1}^{N}$ in $\left[ 0,1\right] ^{d}$ such that for all convex sets $%
\Omega $ contained in $\Sigma $,
\begin{equation*}
\left\vert N^{-1}\sum_{j=1}^{N}\left( f\chi _{\Omega }\right) \left(
x_{j}\right) -\int_{\Omega }f\left( x\right) dx\right\vert \leq c\varepsilon
^{-d}N^{-2/(d+1)}\log ^{\gamma }\left( N\right) .
\end{equation*}

This agrees with the result in \cite{Har10}. However, in the case $\Omega
=\Sigma $, Corollary \ref{6} gives the better estimate
\begin{equation*}
\left\vert N^{-1}\sum_{j=1}^{N}\left( f\chi _{\Sigma }\right) \left(
x_{j}\right) -\int_{\Sigma }f\left( x\right) dx\right\vert \leq c\varepsilon
^{-d}N^{-1}\log ^{d}\left( N\right) .
\end{equation*}

\bigskip

2010 Mathematics Subject Classification. Primary 41A55; Secondary 11K38.

Key words and phrases. Koksma Hlawka inequality, Quadrature, discrepancy,
harmonic analysis.

\bigskip

\textbf{Luca Brandolini}

Dipartimento di Ingegneria dell'Informazione e Metodi Matematici,

Universit\`{a} di Bergamo,

Viale Marconi 5, 24044 Dalmine, Bergamo, Italia.

\textit{luca.brandolini@unibg.it}

\textbf{Leonardo Colzani}

Dipartimento di Matematica e Applicazioni, Edificio U5,

Universit\`{a} di Milano Bicocca,

Via R.Cozzi 53, 20125 Milano, Italia.

\textit{leonardo.colzani@unimib.it}

\textbf{Giacomo Gigante}

Dipartimento di Ingegneria dell'Informazione e Metodi Matematici,

Universit\`{a} di Bergamo,

Viale Marconi 5, 24044 Dalmine, Bergamo, Italia.

\textit{giacomo.gigante@unibg.it}

\textbf{Giancarlo Travaglini}

Dipartimento di Statistica, Edificio U7,

Universit\`{a} di Milano-Bicocca,

Via Bicocca degli Arcimboldi 8, 20126 Milano, Italia.

\textit{giancarlo.travaglini@unimib.it}

\newpage

\end{document}